\documentclass{amsart}
\newtheorem{thm}{Theorem}

\theoremstyle{definition}

\newtheorem{ex}[thm]{Example}

\providecommand{\abs}[1]{\lvert#1\rvert}

\begin{document}
\title[Anscombe theorem]{An Anscombe-type theorem}

\author{Patrizia Berti}
\address{Patrizia Berti, Dipartimento di Matematica Pura ed Applicata ''G. Vitali'', Universita' di Modena e Reggio-Emilia, via Campi 213/B, 41100 Modena, Italy}
\email{patrizia.berti@unimore.it}

\author{Irene Crimaldi}
\address{Irene Crimaldi, IMT Institute for Advanced Studies, Piazza San Ponziano 6, 55100 Lucca, Italy}
\email{irene.crimaldi@imtlucca.it}

\author{Luca Pratelli}
\address{Luca Pratelli, Accademia Navale, viale Italia 72, 57100 Livorno,
Italy} \email{pratel@mail.dm.unipi.it}

\author{Pietro Rigo}
\address{Pietro Rigo (corresponding author), Dipartimento di Matematica ''F. Casorati'', Universita' di Pavia, via Ferrata 1, 27100 Pavia, Italy}
\email{pietro.rigo@unipv.it}


\begin{abstract} Let $(X_n)$ be a sequence of random variables (with values in a separable metric space) and $(N_n)$ a sequence of random indices. Conditions for $X_{N_n}$ to converge stably (in particular, in distribution) are provided. Some examples, where such conditions work but those already existing fail, are given as well.
\end{abstract}

\maketitle

\section{Introduction}\label{intro}

Anscombe's theorem (AT) gives conditions for $X_{N_n}$ to converge in distribution, where $(X_n)$ is a sequence of random variables and $(N_n)$ a sequence of random indices. Roughly speaking, such conditions are: (i) $N_n\rightarrow\infty$ in some sense; (ii) $X_n$ converges in distribution; (iii) For large $n$, $X_j$ is close to $X_n$ provided $j$ is close to $n$. (Precise definitions are given in Subsection \ref{classical}).

In particular, in AT, condition (i) is realized as

\vspace{0.2cm}

\begin{itemize}

\item[(a)] $N_n/k_n\overset{P}\longrightarrow u$, where $k_n>0$ and $u>0$ are constants and $k_n\rightarrow\infty$.

\end{itemize}

\vspace{0.2cm}

\noindent Under (a), it is very hard to improve on AT. The only possibility is to look for some optimal form of condition (iii). See e.g. \cite{MNRT}.

But condition (a) is often generalized into

\vspace{0.2cm}

\begin{itemize}

\item[(a*)] $N_n/k_n\overset{P}\longrightarrow U$, where $U>0$ is a random variable.

\end{itemize}

\vspace{0.2cm}

\noindent For instance, condition (a*) suffices for $X_{N_n}$ to converge in distribution in case $X_n=n^{-1/2}\sum_{i=1}^n\bigl\{Z_i-E(Z_1)\bigr\}$, where $(Z_n)$ is an i.i.d. sequence with $E(Z_1^2)<\infty$. However, under (a*), convergence in distribution of $X_n$ is not enough. To get converge in distribution of $X_{N_n}$, condition (ii) is to be strengthened.

One natural solution is to request {\em stable} convergence of $X_n$. This is made precise by a result of Zhang Bo \cite{ZHANG} (Theorem \ref{n49v56} in the sequel). According to Theorem \ref{n49v56}, $X_{N_n}$ converges stably (in particular, in distribution) provided $X_n$ converges stably, condition (a*) holds, and some form of (iii) is satisfied. The statement of (iii) depends on whether $U$ is, or it is not, discrete.

In this paper, Theorem \ref{n49v56} is (strictly) improved. Our main result (Theorem \ref{b56th9k} in the sequel) has two possible merits. It does not depend on whether $U$ is discrete. And, more importantly, it requests a form of (iii) weaker than the corresponding one in Theorem \ref{n49v56}. Indeed, in Theorem \ref{n49v56}, the asked version of (iii) does not involve the $N_n$. As a consequence, it potentially works for {\em every} sequence $(N_n)$ of random times but it is also rather strong. Instead, in Theorem \ref{b56th9k}, we exploit a form of (iii) which is tailor-made on the particular sequence of random times at hand.

A few examples, where Theorem \ref{b56th9k} works but Theorem \ref{n49v56} fails, are given as well. We mention Examples \ref{birgikjh} and \ref{costolina} concerning the exchangeable CLT and the exchangeable empirical process.

\section{Stable convergence}\label{stable}

Let $\mathcal{X}$ be a metric space and $(\Omega,\mathcal{A},P)$ a probability space. A {\em kernel} (or a {\em random
probability measure}) on $\mathcal{X}$ is a map $K$ on $\Omega$ such that:

\begin{itemize}

\item[$-$] $K(\omega)$ is a Borel probability measure on $\mathcal{X}$ for each
$\omega\in\Omega$;

\item[$-$] $\omega\mapsto K(\omega)(B)$ is $\mathcal{A}$-measurable for each Borel set $B\subset \mathcal{X}$.

\end{itemize}

\noindent For every bounded Borel function $f:\mathcal{X}\rightarrow\mathbb{R}$, we let $K(f)$ denote the real random variable
\begin{equation*}
K(\omega)(f)=\int f(x)\,K(\omega)(dx).
\end{equation*}

Let $(X_n)$ be a sequence of $\mathcal{X}$-valued random variables on $(\Omega,\mathcal{A},P)$. Given a Borel probability measure $\mu$ on $\mathcal{X}$, say that $X_n$ converges in distribution to $\mu$ if $\mu(f)=\lim_nE\bigl\{f(X_n)\bigr\}$ for all bounded continuous functions $f:\mathcal{X}\rightarrow\mathbb{R}$. In this case, we also write $X_n\overset{d}\longrightarrow X$ for any $\mathcal{X}$-valued random variable $X$ with distribution $\mu$. Next, let $\mathcal{G}\subset\mathcal{A}$ be a sub-$\sigma$-field and $K$ a kernel on $\mathcal{X}$. Say that $X_n$ {\em converges $\mathcal{G}$-stably to} $K$ if
\begin{equation*}
E\bigl\{K(f)\mid H\bigr\}=\lim_nE\bigl\{f(X_n)\mid H\bigr\}
\end{equation*}
for all $H\in\mathcal{G}$ with $P(H)>0$ and all bounded continuous $f:\mathcal{X}\rightarrow\mathbb{R}$.

$\mathcal{G}$-stable convergence always implies convergence in distribution (just let $H=\Omega$). Further, it reduces to convergence in distribution for $\mathcal{G}=\{\emptyset,\Omega\}$ and is connected to convergence in probability for $\mathcal{G}=\mathcal{A}$. Suppose in fact $\mathcal{X}$ is separable and take an $\mathcal{X}$-valued random variable $X$ on $(\Omega,\mathcal{A},P)$. Then, $X_n\overset{P}\longrightarrow X$ if and only if $X_n$ converges $\mathcal{A}$-stably to the kernel $K=\delta_X$.

We refer to \cite{CLP} and references therein for more on stable convergence.

\section{Results}

\subsection{Notation}

All random variables appearing in the sequel, unless otherwise stated, are defined on a fixed probability space $(\Omega,\mathcal{A},P)$.

Let $(S,d)$ be a separable metric space. The basic ingredients are three sequences
\begin{gather*}
(X_n:n\geq 0),\quad (N_n:n\geq 0),\quad (k_n:n\geq 0),
\end{gather*}
where the $X_n$ are $S$-valued random variables, the $N_n$ are random times (i.e., random variables with values in $\{0,1,2,\ldots\}$) and the $k_n$ are strictly positive constants such that $k_n\rightarrow\infty$. We let
\begin{gather*}
M_n(\delta)=\max_{j:\abs{n-j}\leq n\,\delta}d(X_j,X_n)
\end{gather*}
for all $n\geq 0$ and $\delta>0$. Finally, $K$ denotes a kernel on $S$.

\subsection{Classical Anscombe's theorem and one of its developments}\label{classical}

Let $\mu$ be a Borel probability measure on $S$. According to AT, for $X_{N_n}$ to converge in distribution to $\mu$, it suffices that

\vspace{0.2cm}

\begin{itemize}

\item[(a)] $N_n/k_n\overset{P}\longrightarrow u$, where $u>0$ is a constant;

\item[(b)] $X_n$ converges in distribution to $\mu$;

\item[(c)] $\inf_{\delta>0}\,\limsup_n\,P\bigl(M_n(\delta)>\epsilon\bigr)=0\,\,$ for all $\epsilon>0$.

\end{itemize}

\vspace{0.2cm}

Soon after its appearance, AT has been investigated and developed in various ways. See e.g. \cite{GK}, \cite{GUT}, \cite{MNRT}, \cite{S}, \cite{ZHANG} and references therein. To our knowledge, most results preserve the structure of the classical AT, for they lead to convergence of $X_{N_n}$ (in distribution or stably) under suitable versions of conditions (a)-(b)-(c). In particular, much attention is paid to possible alternative versions of condition (c). Also, as remarked in Section \ref{intro}, condition (a) is often generalized into

\vspace{0.2cm}

\begin{itemize}

\item[(a*)] $N_n/k_n\overset{P}\longrightarrow U$, where $U>0$ is a random variable.

\end{itemize}

\vspace{0.2cm}

Replacing (a) with (a*) is not free but implies strengthening (b) and/or (c). A remarkable example is the following. In the sequel, $U$ denotes a real random variable and $\mathcal{G}$ a sub-$\sigma$-field of $\mathcal{A}$ such that
\begin{gather*}
U>0\quad\text{and}\quad\sigma(U)\subset\mathcal{G}.
\end{gather*}

\begin{thm}\label{n49v56} {\bf (Zhang Bo \cite{ZHANG}).} Let $U$ be strictly positive and $\mathcal{G}$-measurable. Suppose condition (a*) holds and

\vspace{0.2cm}

\begin{itemize}

\item[(b*)] $X_n$ converges $\mathcal{G}$-stably to $K$.

\end{itemize}

\vspace{0.2cm}

\noindent Then, $X_{N_n}$ converges $\mathcal{G}$-stably to $K$ provided condition (c) holds and $U$ is discrete. Or else, $X_{N_n}$ converges $\mathcal{G}$-stably to $K$ provided

\vspace{0.2cm}

\begin{itemize}

\item[(c*)] For each $\epsilon>0$, there is $\delta>0$ such that

\begin{gather*}
\limsup_n\,P\bigl(M_n(\delta)>\epsilon\mid H\bigr)<\epsilon\quad\text{for all }H\in\mathcal{G}\text{ with }P(H)>0.
\end{gather*}

\end{itemize}

\vspace{0.2cm}

\end{thm}

Theorem \ref{n49v56} is our starting point. Roughly speaking, it can be summarized as follows. Suppose (a*) and (c) hold but (a) fails. If $U$ is discrete, $X_{N_n}$ still converges in distribution (in fact, it converges stably) up to replacing (b) with (b*). If $U$ is not discrete, instead, condition (c) should be strengthened as well.

\subsection{Improving Theorem \ref{n49v56}}

Suppose conditions (a*)-(b*) hold but $U$ is not necessarily discrete. As implicit in Theorem \ref{n49v56}, it may be that (c) holds and yet $X_{N_n}$ fails to converge $\mathcal{G}$-stably to $K$; see Example \ref{marinatel}. Hence, to get $X_{N_n}\overset{\mathcal{G}-stably}\longrightarrow K$, condition (c) is to be modified. Plainly, a number of conditions could serve to this purpose. We now investigate two of them.

One (crude) possibility is just replacing $n$ with $N_n$ in condition (c), that is,

\vspace{0.2cm}

\begin{itemize}

\item[(d)] $\inf_{\delta>0}\,\limsup_n\,P\bigl(M_{N_n}(\delta)>\epsilon\bigr)=0\,\,$ for all $\epsilon>0$,

\end{itemize}

\vspace{0.2cm}

\noindent where

\begin{gather*}
M_{N_n}(\delta)=\max_{j:\abs{N_n-j}\leq N_n\,\delta}d(X_j,X_{N_n}).
\end{gather*}

\noindent Unlike condition (c*) of Theorem \ref{n49v56}, which works for {\em every} sequence $N_n$ (as far as (a*) and (b*) are satisfied), condition (d) is tailor-made on the particular sequence of random times at hand.

In view of (a*), another option is replacing $M_n(\delta)$ with
\begin{gather*}
M_{[k_n\,U]}(\delta)=\max_{j:\abs{[k_n\,U]-j}\leq [k_n\,U]\,\delta}d(X_j,X_{[k_n\,U]}).
\end{gather*}
The corresponding condition is

\vspace{0.2cm}

\begin{itemize}

\item[(e)] $\inf_{\delta>0}\,\limsup_n\,P\bigl(M_{[k_n\,U]}(\delta)>\epsilon\bigr)=0\,\,$ for all $\epsilon>0$.

\end{itemize}

\vspace{0.2cm}

Conditions (d) and (e) are actually equivalent. More importantly, they lead to the desired conclusion.

\begin{thm}\label{b56th9k} Let $U$ be strictly positive and $\mathcal{G}$-measurable. Conditions (d) and (e) are equivalent under (a*). Moreover,
\begin{gather*}
X_{N_n}\overset{\mathcal{G}-stably}\longrightarrow K\quad\text{and}\quad X_{[k_n\,U]}\overset{\mathcal{G}-stably}\longrightarrow K
\end{gather*}
under conditions (a*)-(b*)-(d) (or equivalently (a*)-(b*)-(e)).
\end{thm}

\begin{proof} Let $R_n=[k_n\,U]$. We first show that (d) and (e) are equivalent under (a*).

Suppose (a*) and (e) hold and fix $\delta\in (0,1]$. If $\abs{R_n-N_n}\leq\delta\,R_n$ and $j$ is such that $\abs{j-N_n}\leq\delta\,N_n$, then
\begin{gather*}
\abs{j-R_n}\leq\abs{j-N_n}+\delta\,R_n\leq\delta\,N_n+\delta\,R_n\leq 2\,\delta\,R_n+\delta\,\abs{R_n-N_n}\leq 3\,\delta\,R_n.
\end{gather*}
Hence, $\abs{R_n-N_n}\leq\delta\,R_n$ implies
\begin{gather*}
M_{N_n}(\delta)\leq d(X_{R_n},X_{N_n})+\max_{j:\abs{j-R_n}\leq 3\delta\,R_n}d(X_j,X_{R_n})\leq 2\,M_{R_n}(3\,\delta).
\end{gather*}
Given $\epsilon>0$, it follows that
\begin{gather*}
P\bigl(M_{N_n}(\delta)>\epsilon\bigr)\leq P\bigl(\abs{R_n-N_n}>\delta\,R_n\bigr)+P\bigl(M_{R_n}(3\,\delta)>\epsilon/2\bigr).
\end{gather*}
By (a*), $N_n/R_n\overset{P}\longrightarrow 1$ so that $\lim_nP\bigl(\abs{R_n-N_n}>\delta\,R_n\bigr)=0$. Therefore,
\begin{gather*}
\limsup_nP\bigl(M_{N_n}(\delta)>\epsilon\bigr)\leq\limsup_nP\bigl(M_{R_n}(3\,\delta)>\epsilon/2\bigr)
\end{gather*}
and condition (d) follows from condition (e). By precisely the same argument, it can be shown that (a*) and (d) imply (e).

Next, assume conditions (a*)-(b*)-(e). Since
\begin{gather*}
d(X_{R_n},X_{N_n})\leq M_{R_n}(\delta)\quad\text{provided }\,\abs{R_n-N_n}\leq\delta\,R_n,
\end{gather*}
conditions (a*) and (e) yield $d(X_{R_n},X_{N_n})\overset{P}\longrightarrow 0$. Thus, it suffices to prove that $X_{R_n}\overset{\mathcal{G}-stably}\longrightarrow K$. To this end, for each $\delta\in (0,1]$, define
\begin{equation*}
U_\delta=\delta\,I_{\{0<U\leq \delta\}}+\sum_{j=1}^\infty j\,\delta\,I_{\{j\,\delta<U\leq (j+1)\,\delta\}}\quad\text{and}\quad R_n(\delta)=[k_n\,U_\delta].
\end{equation*}

Since $U_\delta$ is discrete, strictly positive and $\mathcal{G}$-measurable, condition (b*) yields $X_{R_n(\delta)}\overset{\mathcal{G}-stably}\longrightarrow K$. Fix in fact $H\in\mathcal{G}$ with $P(H)>0$ and let $H_j=H\cap\{U_\delta=j\,\delta\}$ for all $j\geq 1$. Then, (b*) implies
\begin{gather*}
\lim_nE\bigl\{f(X_{R_n(\delta)})\mid H\bigr\}=\lim_n\sum_jE\bigl\{f(X_{[k_n\,j\,\delta]})\mid H_j\bigr\}\,P(H_j\mid H)
\\=\sum_jE\bigl\{K(f)\mid H_j\bigr\}\,P(H_j\mid H)=E\bigl\{K(f)\mid H\bigr\}
\end{gather*}
for each bounded continuous $f$, where the sum is over those $j$ such that $P(H_j)>0$.

Note also that, on the set $\{U>\delta\}$, one obtains
\begin{gather*}
\abs{R_n-R_n({\delta^2})}=R_n-R_n({\delta^2})=R_n\,\frac{[k_n\,U]-[k_n\,U_{\delta^2}]}{[k_n\,U]}
\\< R_n\,\frac{k_n\,(U-U_{\delta^2})+1}{k_n\,U-1}< R_n\,\frac{k_n\,\delta^2+1}{k_n\,\delta-1}< 2\,\delta\,R_n\quad\text{for large }n.
\end{gather*}
Thus, for $\epsilon>0$ and large $n$,
\begin{gather*}
P\Bigl(\,d(X_{R_n}\,,\,X_{R_n({\delta^2})}\,)>\epsilon\Bigr)
\leq P(U\leq\delta)+P\bigl(M_{R_n}(2\,\delta)>\epsilon\bigr).
\end{gather*}
By condition (e) and since $U>0$, it follows that
\begin{gather}\label{pm7y}
\inf_{\delta>0}\limsup_nP\Bigl(\,d(X_{R_n}\,,\,X_{R_n({\delta^2})}\,)>\epsilon\Bigr)=0.
\end{gather}

Finally, fix $\epsilon>0$, $H\in\mathcal{G}$ with $P(H)>0$, and a closed set $C\subset S$. Let $C_\epsilon=\{x\in S:d(x,C)\leq\epsilon\}$. By \eqref{pm7y}, there is $\delta\in (0,1]$ such that
\begin{gather*}
\limsup_nP\Bigl(\,d(X_{R_n}\,,\,X_{R_n({\delta^2})}\,)>\epsilon\Bigr)<\epsilon\,P(H).
\end{gather*}
With such a $\delta$, since $X_{R_n(\delta^2)}\overset{\mathcal{G}-stably}\longrightarrow K$, one obtains
\begin{gather*}
\limsup_nP\bigl(X_{R_n}\in C\mid H\bigr)\leq\limsup_n\Bigl\{P\Bigl(\,d(X_{R_n}\,,\,X_{R_n({\delta^2})}\,)>\epsilon\mid H\Bigr)+P\Bigl(X_{R_n(\delta^2)}\in C_\epsilon\mid H\Bigr)\Bigr\}
\\<\epsilon+\limsup_nP\Bigl(X_{R_n(\delta^2)}\in C_\epsilon\mid H\Bigr)\leq\epsilon+E\bigl\{K(C_\epsilon)\mid H\bigr\}.
\end{gather*}
As $\epsilon\rightarrow 0$, it follows that $\limsup_nP\bigl(X_{R_n}\in C\mid H\bigr)\leq E\bigl\{K(C)\mid H\bigr\}$. Therefore, $X_{R_n}\overset{\mathcal{G}-stably}\longrightarrow K$ and this concludes the proof.

\end{proof}

Theorem \ref{b56th9k} unifies the two parts of Theorem \ref{n49v56} ($U$ discrete and $U$ not discrete). In addition, Theorem \ref{b56th9k} strictly improves Theorem \ref{n49v56}. In fact, condition (c*) implies condition (e) but not conversely. Two (natural) examples where (e) holds and (c*) fails are given in the next section; see Examples \ref{birgi} and \ref{birgikjh}. Here, we prove the direct implication.

\begin{thm}\label{rrpm}
Let $U$ be strictly positive and $\mathcal{G}$-measurable. If condition (c) holds and $U$ is discrete, or if condition (c*) holds, then condition (e) holds.
\end{thm}

\begin{proof}
Let $R_n=[k_n\,U]$. Suppose (c) holds and $U$ is discrete. Then it suffices to note that, for each $\epsilon>0$ and $u>0$ such that $P(U=u)>0$, one obtains
\begin{gather*}
\limsup_n\,P\bigl(M_{R_n}(\delta)>\epsilon\mid U=u\bigr)
=\limsup_n\,P\bigl(M_{[k_n\,u]}(\delta)>\epsilon\mid U=u\bigr)
\\\leq P(U=u)^{-1}\,\limsup_nP\bigl(M_n(\delta)>\epsilon\bigr)\longrightarrow 0\quad\text{as }\delta\rightarrow 0.
\end{gather*}
Next, suppose (c*) holds. Given $\epsilon>0$, take $\delta>0$ such that
\begin{gather*}
\limsup_n\,P\bigl(M_n(\delta)>\epsilon/2\mid H\bigr)<\epsilon/2\quad\text{for all }H\in\mathcal{G}\text{ with }P(H)>0.
\end{gather*}
Fix $u,\,\gamma >0$ and define $H=\{u-\gamma\leq U<u+\gamma\}$. Take $j$ and $n$ such that
\begin{gather*}
\abs{j-R_n}\leq (\delta/4)\,R_n,\quad k_n\,\gamma>1,\quad k_n\,u<2\,[k_n\,u].
\end{gather*}
On the set $H$, one obtains
\begin{gather*}
\abs{j-[k_n\,u]}\leq\abs{j-R_n}+\abs{R_n-[k_n\,u]}\leq (\delta/4)\,R_n+\abs{[k_n\,U]-[k_n\,u]}
\\<(\delta/4)\,k_n\,(u+\gamma)+k_n\,\gamma+1<[k_n\,u]\,\frac{2}{u}\,\{(\delta/4)\,(u+\gamma)+2\,\gamma\}.
\end{gather*}
Letting $\delta^*=(2/u)\,\bigl\{(\delta/4)\,(u+\gamma)+2\,\gamma\bigr\}$, it follows that
\begin{equation*}
M_{R_n}(\delta/4)\leq M_{[k_n\,u]}(\delta^*)+d\bigl(X_{R_n}\,,\,X_{[k_n\,u]}\bigr)\leq 2\,M_{[k_n\,u]}(\delta^*)
\end{equation*}
on $H$ for large $n$. Since $H\in\mathcal{G}$,
\begin{gather*}
\limsup_nP\bigl(M_{R_n}(\delta/4)>\epsilon\mid H\bigr)\leq\limsup_nP\bigl(M_{[k_n\,u]}(\delta^*)>\epsilon/2\mid H\bigr)
\\\leq\limsup_nP\bigl(M_n(\delta^*)>\epsilon/2\mid H\bigr)<\epsilon/2
\end{gather*}
provided $P(H)>0$ and $u,\,\gamma$ are such that $\delta^*\leq\delta$, or equivalently
\begin{equation*}
\frac{\gamma}{u}\leq\frac{\delta}{8+\delta}.
\end{equation*}
Finally, take $0<a<b$ such that $P(a\leq U<b)>1-(\epsilon/2)$. The set $\{a\leq U<b\}$ can be partitioned into sets $H_i=\{u_i-\gamma\leq U<u_i+\gamma\}$ such that $(\gamma/a)\leq\delta/(8+\delta)$ and $u_1=a+\gamma<u_2<\ldots$. On noting that $(\gamma/u_i)\leq\delta/(8+\delta)$ for all $i$,
\begin{gather*}
\limsup_nP\bigl(M_{R_n}(\delta/4)>\epsilon\bigr)<\epsilon/2+\limsup_nP\bigl(M_{R_n}(\delta/4)>\epsilon,\,a\leq U<b\bigr)
\\\leq\epsilon/2+\sum_i\limsup_nP\bigl(M_{R_n}(\delta/4)>\epsilon\mid H_i\bigr)\,P(H_i)<\epsilon
\end{gather*}
where the sum is over those $i$ with $P(H_i)>0$. This concludes the proof.
\end{proof}

\section{Examples}

It is implicit in Theorem \ref{n49v56} that, when $U$ is not discrete, conditions (a*)-(b*)-(c) are not enough for $X_{N_n}\overset{\mathcal{G}-stably}\longrightarrow K$ (where $K$ is the kernel involved in condition (b*)). However, we do not know of any explicit example. So, we begin with one such example.

\begin{ex} {\bf(Conditions (a*)-(b*)-(c) do not imply $X_{N_n}\overset{\mathcal{G}-stably}\longrightarrow K$).}\label{marinatel}
Let $\Omega=[0,1)$, $\mathcal{A}$ the Borel $\sigma$-field and $P$ the Lebesgue measure. For each $n\geq 1$, define
 \begin{gather*}
A_n=\bigl[\log n,\,\log (n+1)\bigr)\,\,\text{ modulo }1,
\end{gather*}
that is, $A_1=[0,\,\log 2)$, $A_2=[\log 2,\,1)\cup [0,\,(\log 3)-1)$ and so on. Define also $X_0=0$ and $X_n=I_{A_n}$ for $n\geq 1$. Since $P(A_n)=\log ((n+1)/n)$, then $X_n\overset{P}\longrightarrow 0$, or equivalently $X_n$ converges $\mathcal{A}$-stably to the point mass at 0 (see Section \ref{stable}). Thus, condition (b*) holds with $\mathcal{G}=\mathcal{A}$ and $K$ the point mass at 0. Given $\epsilon>0$,
\begin{gather*}
P\bigl(M_n(\delta)>\epsilon,\,X_n=0\bigr)\leq P\Bigl(\bigcup_{j:\abs{n-j}\leq n\,\delta}A_j\Bigr)
\leq \sum_{j:\abs{n-j}\leq n\,\delta}P(A_j)\leq\log\frac{[n\,(1+\delta)]+1}{[n\,(1-\delta)]}.
\end{gather*}
Since $P(X_n=0)\rightarrow 1$, it follows that
\begin{gather*}
\limsup_nP\bigl(M_n(\delta)>\epsilon\bigr)=\limsup_nP\bigl(M_n(\delta)>\epsilon,\,X_n=0\bigr)\leq\log\frac{1+\delta}{1-\delta},
\end{gather*}
that is, condition (c) holds. Finally, define $U(\omega)=\exp{(\omega)}$ for all $\omega\in [0,1)$ and
\begin{gather*}
N_n=[U\,\exp{(r_n)}],
\end{gather*}
where the $r_n$ are non-negative integers such that $r_n\rightarrow\infty$. Condition (a*) is trivially true. Further, for each $n$, one obtains $\{N_n=k\}\subset A_k$ for all $k$, so that $X_{N_n}=1$. Thus, $X_{N_n}$ fails to converge $\mathcal{A}$-stably to the point mass at 0.
\end{ex}

We next prove that condition (e) does not imply condition (c*). We give two examples. The first is just a modification of Example \ref{marinatel}, while the second (which requires some more calculations) concerns the exchangeable CLT. Recall that (d) and (e) are equivalent under (a*).

\begin{ex} {\bf (Example \ref{marinatel} revisited).}\label{birgi}
Conditions (b*)-(c)-(c*) depend on $(X_n)$ and $\mathcal{G}$ only. In view of Theorem \ref{n49v56}, condition (c*) fails in Example \ref{marinatel}. Hence, to build an example where (c*) fails but (a*)-(b*)-(c)-(d) hold, it suffices to suitably modify the random times $N_n$ of Example \ref{marinatel}. Precisely, suppose $(\Omega,\mathcal{A},P)$, $U$, $(X_n)$ and $\mathcal{G}$ are as in Example \ref{marinatel}, but the random times are now
\begin{gather*}
N_n=\Bigl[\,\frac{T_{n-1}+T_n}{2}\,\Bigr]\quad\text{where }\,T_n=\inf\{j:j>T_{n-1}\text{ and }X_j=1\}\text{ and }N_0=T_0=0.
\end{gather*}
Then, (c*) fails while (b*)-(c) hold. It is not hard to see that $T_n=[\exp{(n-1)}\,U]$ for $n\geq 1$. Thus, conditions (a*) and (d) are both trivially true. (As to (d), just note that $T_{n-1}<N_n\,(1-\delta)<N_n\,(1+\delta)<T_n$ for large $n$ and small $\delta$).
\end{ex}

\begin{ex} {\bf (Exchangeable CLT).}\label{birgikjh}
Let $(Z_n:n\geq 1)$ be an exchangeable sequence of real random variables with tail $\sigma$-field $\mathcal{T}$. By de Finetti's theorem, $(Z_n)$ is i.i.d. conditionally on $\mathcal{T}$. Basing on this fact, if $E(Z_1^2)<\infty$, it is not hard to see that
\begin{gather*}
\frac{\sum_{i=1}^n\{Z_i-E(Z_1\mid\mathcal{T})\}}{\sqrt{n}}\overset{\mathcal{A}-stably}\longrightarrow N(0,L)
\end{gather*}
where $L=E(Z_1^2\mid\mathcal{T})-E(Z_1\mid\mathcal{T})^2$ and $N(0,\sigma^2)$ denotes the Gaussian law with mean 0 and variance $\sigma^2$ (with $N(0,0)$ the point mass at 0); see e.g. Theorem 3.1 of \cite{BPR04} and the subsequent remark. Fix a $\mathcal{T}$-measurable random variable $U>0$ and define
\begin{gather*}
N_n=[n\,U],\quad X_0=0,\quad X_n=\frac{\sum_{i=1}^n\{Z_i-E(Z_1\mid\mathcal{T})\}}{\sqrt{n}}.
\end{gather*}
Then, conditions (a*)-(b*)-(c)-(d) are satisfied (with $\mathcal{G}=\mathcal{A}$ and $K=N(0,L)$) so that
\begin{gather*}
\frac{\sum_{i=1}^{N_n}\{Z_i-E(Z_1\mid\mathcal{T})\}}{\sqrt{N_n}}\overset{\mathcal{A}-stably}\longrightarrow N(0,L)
\end{gather*}
because of Theorem \ref{b56th9k}. Indeed, (a*)-(b*) are obvious and (c) can be checked precisely as (d). As to (d), given $\epsilon>0$, just note that
\begin{gather*}
\limsup_n\,P\bigl(M_{N_n}(\delta)>\epsilon\mid\mathcal{T}\bigr)\leq\limsup_n\,P\bigl(M_n(\delta)>\epsilon\mid\mathcal{T}\bigr)\quad\text{a.s.}
\end{gather*}
for $N_n$ is $\mathcal{T}$-measurable, and
\begin{gather*}
\limsup_n\,P\bigl(M_n(\delta)>\epsilon\mid\mathcal{T}\bigr)\overset{a.s.}\longrightarrow 0\quad\text{as }\delta\rightarrow 0
\end{gather*}
for $(Z_n)$ is i.i.d. conditionally on $\mathcal{T}$. Thus,
\begin{gather*}
\limsup_n\,P\bigl(M_{N_n}(\delta)>\epsilon\bigr)\leq\int\limsup_n\,P\bigl(M_{N_n}(\delta)>\epsilon\mid\mathcal{T}\bigr)\,dP
\\\leq\int\limsup_n\,P\bigl(M_n(\delta)>\epsilon\mid\mathcal{T}\bigr)\,dP\longrightarrow 0\quad\text{as }\delta\rightarrow 0.
\end{gather*}

It remains to see that condition (c*) may fail. We verify this fact for
\begin{gather*}
\mathcal{G}=\sigma(U)\quad\text{and}\quad Z_n=U\,V_n
\end{gather*}
where
\begin{itemize}

\item $U$ is any random variable such that $U>0$, $E(U^2)<\infty$ and $P(U>u)>0$ for all $u>0$;

\item $(V_n)$ is i.i.d., $V_1\sim N(0,1)$, and $(V_n)$ is independent of $U$.

\end{itemize}
Such a sequence $(Z_n)$ is exchangeable and $E(Z_1^2)=E(U^2)<\infty$. Furthermore, $E(Z_1\mid\mathcal{T})=0$ a.s. and $U$ is $\mathcal{T}$-measurable (up to modifications on $P$-null sets) for
\begin{gather*}
\frac{\sum_{i=1}^nZ_i}{n}=U\,\frac{\sum_{i=1}^nV_i}{n}\overset{a.s.}\longrightarrow 0\quad\text{and}\quad
\frac{\sum_{i=1}^nZ_i^2}{n}=U^2\,\frac{\sum_{i=1}^nV_i^2}{n}\overset{a.s.}\longrightarrow U^2.
\end{gather*}
Next, a direct calculation shows that
\begin{gather*}
\frac{\sum_{i=1}^nV_i}{\sqrt{n}}-\frac{\sum_{i=1}^mV_i}{\sqrt{m}}\sim N\bigl(0,\,2-2\,\sqrt{n/m}\,\bigr)\quad\text{for }1\leq n\leq m.
\end{gather*}
Thus, conditionally on $U$,
\begin{gather*}
X_n-X_{[n\,(1-\delta)]}=U\,\Bigl\{\frac{\sum_{i=1}^nV_i}{\sqrt{n}}-\frac{\sum_{i=1}^{[n\,(1-\delta)]}V_i}{\sqrt{[n\,(1-\delta)]}}\,\Bigr\}\sim N\bigl(0,\,U^2\,\sigma^2_n(\delta)\bigr)
\end{gather*}
where $\delta\in (0,1)$ and
\begin{gather*}
\sigma^2_n(\delta)=2-2\,\sqrt{\frac{[n\,(1-\delta)]}{n}}\geq 2-2\,\sqrt{1-\delta}.
\end{gather*}
Define $H=\{U>u\}$ and $f(\delta)=2\,\sqrt{2-2\,\sqrt{1-\delta}}$ for some $u>0$ and $\delta\in (0,\,1/2)$. Letting $\Phi$ denote the standard normal distribution function, for each $n$ such that $n-[n\,(1-\delta)]\leq n\,2\,\delta$, one obtains
\begin{gather*}
P\bigl(M_n(2\,\delta)>1/2\mid H\bigr)\geq P\Bigl(\abs{X_n-X_{[n\,(1-\delta)]}}>1/2\mid H\Bigr)
\\=P(H)^{-1}\,\int_H P\Bigl(\abs{X_n-X_{[n\,(1-\delta)]}}>1/2\mid U\Bigr)\,dP
\\=P(H)^{-1}\,\int_H 2\,\Phi\bigl(-\frac{1}{2\,U\,\sigma_n(\delta)}\bigr)\,dP
\\\geq 2\,P(H)^{-1}\,\int_H \Phi\bigl(-\frac{1}{U\,f(\delta)}\bigr)\,dP\geq 2\,\Phi\bigl(-\frac{1}{u\,f(\delta)}\bigr).
\end{gather*}
Since $P(U>u)>0$ for all $u>0$, condition (c*) (applied with $\epsilon=1/2$) would imply $\Phi\bigl(-\frac{1}{u\,f(\delta)}\bigr)<1/4$ for some fixed $\delta$ and all $u>0$. But this is absurd for $\lim_{u\rightarrow\infty}\Phi\bigl(-\frac{1}{u\,f(\delta)}\bigr)=\Phi(0)=1/2$. Therefore, (c*) fails in this example.
\end{ex}

Our last example deals with empirical processes for non independent data. Let $l^\infty(\mathbb{R})$ denote the space of real bounded functions on $\mathbb{R}$ equipped with uniform distance.

\begin{ex} {\bf (Exchangeable empirical processes).}\label{costolina} Again, let $(Z_n:n\geq 1)$ be an exchangeable sequence of real random variables with tail $\sigma$-field $\mathcal{T}$. Let $F$ be a random distribution function satisfying
\begin{gather*}
F(t)=P(Z_1\leq t\mid\mathcal{T})\quad\text{a.s. for all }t\in\mathbb{R}.
\end{gather*}
The $n$-th empirical process can be defined as
\begin{gather*}
X_n(t)=\sqrt{n}\,\Bigl\{(1/n)\,\sum_{i=1}^nI_{\{Z_i\leq t\}}-F(t)\Bigr\}\quad\text{for }t\in\mathbb{R}.
\end{gather*}
Define also the process $X(t)=\mathbb{B}\bigl(F(t)\bigr)$, $t\in\mathbb{R}$, where $\mathbb{B}$ is a Brownian-bridge process independent of $F$. (Such a $\mathbb{B}$ is available up to enlarging the basic probability space $(\Omega,\mathcal{A},P)$). If $P(Z_1=Z_2)=0$ or if $Z_1$ is discrete, then $X_n\overset{d}\longrightarrow X$ in the metric space $l^\infty(\mathbb{R})$; see \cite{BPR04}-\cite{BPR06} for details. But $l^\infty(\mathbb{R})$ is not separable and working with it yields various measurability issues. So, to avoid technicalities, we assume $0\leq Z_1\leq 1$ and we take $S$ to be the space of real cadlag functions on $[0,1]$ equipped with Skorohod distance. Then, $X_n\overset{d}\longrightarrow X$ in the separable metric space $S$; see e.g. Theorem 3 of \cite{BPR06}. Actually, basing on de Finetti's theorem, it can be shown that $X_n$ converges $\mathcal{A}$-stably to a certain kernel $K$ on $S$. Precisely, for each distribution function $H$, let $Q_H$ denote the probability distribution (on the Borel sets of $S$) of the process $X_H(t)=\mathbb{B}\bigl(H(t)\bigr)$, $t\in [0,1]$. Then, $K$ can be written as
\begin{gather*}
K(A)=Q_F(A)\quad\text{for all Borel sets }A\subset S.
\end{gather*}
Finally, let $N_n=[n\,U]$ where $U>0$ is any $\mathcal{T}$-measurable random variable. Then, condition (a*) is trivially true, (b*) holds with $\mathcal{G}=\mathcal{A}$, and (d) can be checked as in Example \ref{birgikjh}. Thus, Theorem \ref{b56th9k} implies $X_{N_n}\overset{\mathcal{A}-stably}\longrightarrow K$. This fact can not be deduced by Theorem \ref{n49v56}, however, for condition (c*) may fail.
\end{ex}


\begin{thebibliography}{99}

\bibitem{BPR04}
Berti P., Pratelli L., Rigo P. (2004) Limit theorems for a class of
identically distributed random variables, {\it Ann. Probab.}, 32, 2029-2052.

\bibitem{BPR06}
Berti P., Pratelli L., Rigo P. (2006) Asymptotic behaviour of the empirical process for exchangeable data, {\it Stoch. Proc. Appl.}, 116, 337-344.

\bibitem{CLP} Crimaldi I., Letta G., Pratelli L. (2007) A strong
form of stable convergence, {\em Seminaire de Probabilites} XL, Lect. Notes in Math., 1899, 203-225.

\bibitem{GK} Gnedenko B.V., Korolev V.Y.  (1996) {\em Random summation:
limit theorems and applications}, CRC Press.

\bibitem{GUT} Gut A. (2011) Anscombe's theorem 60 years later, {\em U.U.D.M. Report} 2011:14.

\bibitem{MNRT} Nzi M., Theodorescu R. (1990) Anscombe's condition revisited, {\it Bull. Polish Acad. Sci. Math.}, 38, 55-60.

\bibitem{S} Silvestrov D.S.  (2004) {\em Limit theorems for randomly stopped stochastic processes}, Springer.

\bibitem{ZHANG} Zhang Bo (2000) Stable convergence of random sequences with random indices, {\em J. Math. Sciences}, 99, 1515-1526.

\end{thebibliography}
\end{document}